\documentclass{amsart}
\usepackage{amssymb}
\usepackage{amsmath,amsthm,amscd,amssymb, mathrsfs}

\usepackage{latexsym}
\usepackage{color}

\numberwithin{equation}{section}

\theoremstyle{plain}
\newtheorem{theorem}{Theorem}[section]
\newtheorem{lemma}[theorem]{Lemma}

\newtheorem{proposition}[theorem]{Proposition}
\newtheorem{conjecture}[theorem]{Conjecture}
\newtheorem{question}[theorem]{Question}

\theoremstyle{definition}
\newtheorem{definition}[theorem]{Definition}

\theoremstyle{remark}
\newtheorem{remark}[theorem]{Remark}

\newtheorem{case[theorem]}{Case}

\newcommand{\nothing}[1]{{}}

\def\cD{{\mathcal D}}

\def\cH{{\mathcal H}}

\def\cL{{\mathcal L}}

\def\({\left(}
\def\){\right)}
\def\[{\left[}
\def\]{\right]}
\def\<{\langle}
\def\>{\rangle}

\def\F{\mathbb{F}}

\def\R{\mathbb{R}}

\title[Size of dot product sets ]{Size of dot product sets determined by pairs of  subsets of vector spaces over  finite fields}
\author{ Doowon Koh and Youngjin Pi}
\date{\today}
\address{Department of Mathematics\\
Chungbuk National University \\
Cheon\-gju city, Chungbuk-Do 361-763 Korea}
\email{koh131@chungbuk.ac.kr}

\address{Department of Mathematics\\
Chungbuk National University \\
Cheon\-gju city, Chungbuk-Do 361-763 Korea}
\email{pi@chungbuk.ac.kr}

\thanks{The first author was supported by  the research grant of the Chungbuk National University in 2012 }


\begin{document}

\begin{abstract}
In this paper we study the cardinality of the dot product set generated by two subsets of vector spaces over finite fields.
We notice that the results on the dot product problems for one set can be simply extended to two sets.
Let $E$ and $F$ be  subsets of the $d$-dimensional vector space $\mathbb F_q^d$ over a finite field $\mathbb F_q$ with $q$ elements.
As a new result, we prove that  if  $E$ and $F$ are subsets of the paraboloid and $|E||F|\geq C q^d$ for some large $C>1,$ then $|\Pi(E,F)|\geq c q$ for some $0<c<1.$
In particular, we find a connection between the size of the dot product set and the number of lines through both the origin and a nonzero point in the given set $E.$
As an application of this observation, we obtain more sharpened results on the generalized dot product set problems.
The discrete Fourier analysis and geometrical observation play a crucial role in proving our results.
\end{abstract}
\maketitle

\section{Introduction}

 How many distinct distances can be determined by a finite subset of $\mathbb R^d$? In 1946, this question was addressed by  Erd\H os \cite{Er46}. 
This problem is well known as the Erd\H os distance problem in the Euclidean space.
More generally,  given $E, F \subset \mathbb R^d $ with $|E|, |F|<\infty$,  one may ask for  the cardinality of the distance set $\Delta(E,F)$ in terms of the sizes of $E$ and $F$,
where $| \cdot |$ denotes the cardinality of a finite set of $\mathbb R^d$ and  the distance set $\Delta(E,F)$ is defined by
$$ \Delta(E,F)=\left\{\sqrt{(x_1-y_1)^2+ \cdots+ (x_d-y_d)^2} : x\in E, y\in F\right\}.$$
If $E=F,$ then we shall write $\Delta(E)$ for $\Delta(E,F).$
The first nontrivial result on this problem was obtained by Erd\H os \cite{Er46}.
He proved that if  $E\subset \mathbb R^d$, then $|\Delta(E)| \geq c |E|^{1/d}$ for some constant $0<c<1$ independent of $|E|.$ 
In addition, he conjectured that for every $\varepsilon>0$ there exists $c_\varepsilon >0$ such that  $|\Delta(E)| \geq c_\varepsilon |E|^{2/d-\varepsilon}.$
The conjecture on the plane was recently solved by Guth and Katz \cite{GK10} but it is still open for higher dimensions (see, for example, \cite{SV08, KT04, SV04}).\\

As a continuous version of the Erd\H os distance problem, the Falconer distance problem has been studied.
The Falconer distance conjecture says that  if    $E $ is a compact subset of $\mathbb R^d, d\geq 2,$ and  the Hausdorff dimension of $E$ is greater than $d/2$, then 
the distance set $\Delta(E)$ has a positive Lebesgue measure. Since this conjecture was first addressed by Falconer \cite{Fa85},  much attention has been given to this problem but it has not been solved for any dimensions.  Using the decay estimate of the Fourier transform on the sphere,   Falconer \cite{Fa85} firstly obtained that 
$$ \mbox{dim}_\cH(E)> \frac{d+1}{2} \Longrightarrow \cL(\Delta(E))>0,$$
where $\mbox{dim}_\cH(E)$ denotes the Hausdorff dimension of $E\subset \R^d$ and $\cL(\Delta(E))$ denotes one-dimensional Lebesgue measure of the distance set $\Delta(E).$
The Falconer's result was generalized by  Mattila  who proved in \cite{Ma87} that for any compact sets $E, F \subset \R^d,$
$$\mbox{dim}_\cH(E)+\mbox{dim}_\cH(F)> d+1 \Longrightarrow \cL(\Delta(E,F))>0.$$
The currently best known results on the Falconer problem are due to Wolff \cite{Wo99} for two dimensions  and Erdo\~{g}an  \cite{Er05} for higher dimensions.
Their results say that if $E\subset \R^d, d\geq 2,$ with $\mbox{dim}_\cH(E) > d/2+1/3,$ then $\cL(\Delta(E))>0.$\\

In recent years,   the Erd\H os-Falconer distance problems  have been reconstructed in the finite field setting.
Let $\mathbb F_q^d$ denote the $d$-dimensional vector space over a finite field $\F_q$ with $q$ elements. 
Throughout the paper, we always assume that the characteristic of $\F_q$ is greater than two.
Given $E, F\subset \F_q^d, d\geq 2, $ the distance set, denoted by $\cD(E,F)$, is defined by
$$ \cD(E,F)=\{ \|x-y\|\in \mathbb F_q : x\in E, y\in F\},$$
where $\| \alpha\|= \alpha_1^2 + \cdots+ \alpha_d^2$ for $\alpha=(\alpha_1, \ldots, \alpha_d) \in \F_q^d.$
We point out that the function $\| \cdot\|$ on $\F_q^d$ is not a standard norm but its image is invariant under the rotations in $\F_q^d.$
The Erd\H os distance problem  in the finite field setting is to find the connection between $|\cD(E,F)|$ and cardinalities of $E, F \subset \F_q^d.$
In the prime field setting,  the Erd\H os distance problem in two dimensions was initially posed and studied  by Bourgain-Katz-Tao \cite{BKT04}.
In 2007, Iosevich and Rudnev \cite{IR07} developed the problem in arbitrary dimensional vector spaces over general finite fields.
Using the Kloosterman sum estimate, Iosevich and Rudnev \cite{IR07} obtained that if $E\subset \mathbb F_q^d$, then 
\begin{equation}\label{ARF} |\cD(E,E)| \gg_c \min \left\{ q, \frac{|E|}{q^{\frac{d-1}{2}}} \right\}. \end{equation}
\begin{remark}Here and throughout this paper, the notation $A\gg_c B$ for $A,B>0$ means that there exists a constant $0<c<1$ depending only on the dimension $d$ such that $A \geq c B.$
On the other hand,  we shall use the notation $A\gg_C B$ to indicate that there exists a sufficient large constant $C>1$ depending only on the dimension $d$ such that $A \geq C B.$ The constants $0<c<1$ and $C>1$ may be changed from one line to another line but they are independent of the size of the underlying finite field $\mathbb F_q.$
We also write  $B\ll_C A$ for  $A\gg_c B. $  $A\sim B$ means that there exist constants $0<c<1, 1<C$ such that $c B\leq A\le CB,$
where $c,C$ depend only on the dimension $d.$ \end{remark}

As a finite field version of the Falconer distance problem,  Iosevich and Rudnev \cite{IR07}  conjectured that if $E\subset \mathbb F_q^d$ with $|E|\gg_C q^{d/2}$, then $|\cD(E,E)|\gg_c q.$
As a corollary of  (\ref{ARF}), they also obtained that  $|\cD(E,E)|\gg_c q$  as long as $|E|\gg_C q^{(d+1)/2}.$
The authors in \cite{HIKR10} constructed arithmetic examples which show that the conjecture by Iosevich and Rudnev is not true in odd dimensions and the exponent $(d+1)/2$ gives a sharp result on the Falconer distance problem in odd dimensional vector spaces over $\mathbb F_q$. However, it has been believed that the conjecture may be true in even dimensions, in part because  the authors in \cite{CEHIK09} recently showed that 
if $E \subset {\mathbb F}_q^2$ with $|E|\gg_C q^{4/3}$, then 
$ |\cD(E,E)| \gg_c q.$ When $d=2$, the exponent $4/3$ is better than the exponent $(d+1)/2$ which gives a sharp exponent in odd dimensions.
This result for dimension two was generalized by  Koh and Shen \cite{KS} who proved that if $E, F\subset \mathbb F_q^2$ with $|E||F|\gg_C q^{8/3},$ then $|\cD(E,F)|\gg_c q.$
In \cite{KS13}, they also stated the following conjecture  which generalizes the conjecture originally stated in \cite{IR07} for even dimensions.
\begin{conjecture}\label{KohCon} Let $d\geq 2$ be an even integer. Suppose $E, F\subset \mathbb F_q^d.$  If $|E||F|\gg_C q^d,$ then $|\cD(E,F)|\gg_c q.$
\end{conjecture}
This conjecture has not been solved but there are some specific sets which yield the conclusion of the conjecture for any dimensions $d\geq 2.$ 
For example, Iosevich and Rudnev \cite{IR07} showed that the conclusion of the conjecture holds if $E=F$ and $E$ is a Salem set.
Here, we recall that a set $E\subset \mathbb F_q^d$ is called a Salem set if $|\widehat{E}(m)| \ll_C \sqrt{|E|} / q^d$ for all $m\in \mathbb F_q^d\setminus \{(0,\ldots, 0)\}.$
Considering the number of vectors determined by two sets $E,F \subset \mathbb F_q^d,$  Koh and Shen \cite{KS} deduced that if one of sets $E,F\subset \mathbb F_q^d$ is a Salem set, then
the conclusion of Conjecture \ref{KohCon} follows for any dimensions $d\geq 2.$\\

By analogy with the distance set $\cD(E,F)$, if $E, F\subset \mathbb F_q^d,$ then one can define a set of dot products as
$$ \Pi(E,F)=\{x\cdot y\in \mathbb F_q: x\in E, y\in F\}.$$
In the case when $E=F\subset \mathbb F_q^d$, the authors in \cite{HIKR10} investigated the cardinality of $|\Pi(E,F)|.$
They proved the following result.
\begin{proposition}\label{ProAlex}
 If $E\subset \mathbb F_q^d$ with $|E|\gg_C q^{(d+1)/2}$, then $|\Pi(E,E)|\gg_c q.$
\end{proposition}
In addition, they provided an example to show that  the exponent $(d+1)/2$ in Proposition \ref{ProAlex} can not be improved on a general set $E.$
However, they made a remarkable observation that if $E$ lies on a unit sphere, then 
Proposition \ref{ProAlex} can be improved. More precisely, they proved the following.
\begin{proposition}\label{Prosphere}
Suppose that $E\subset S_1:=\{x\in \mathbb F_q^d: x_1^2+ \cdots+x_d^2=1\}$. If  $|E|\gg_C q^{d/2},$ then   $|\Pi(E,E)|\gg_c q.$
\end{proposition}
As a direct application of this proposition, they deduced the following  Erd\H os-Falconer distance result on the unit sphere.
\begin{proposition}\label{SphereDistance} Let $S_1=\{x\in \mathbb F_q^d: x_1^2+x_2^2+ \cdots + x_d^2=1\}.$  If $d\geq 3,$ $E\subset S_1$ and $|E|\gg_C q^{d/2},$ then  $|\cD(E,E)|\gg_c q.$
\end{proposition}
This proposition implies that the conclusion of Conjecture \ref{KohCon} holds for the dimensions $d\geq 3$ if the set $E$ is restricted to the unit sphere.\\


\subsection{Purpose of this paper}
For each $x\in \mathbb F_q^{d*},$ define
\begin{equation}\label{linedef}\l_x=\{sx\in \mathbb F_q^d:  s\in \mathbb F_q^*\},\end{equation}
where  we denote $\mathbb F_q^{d*}=\mathbb F_q^d\setminus \{(0,\ldots,0)\}$ for $d\geq 2,$ and 
$\mathbb F_q^*=\mathbb F_q\setminus \{0\}.$ Estimating $\max_{x\in \mathbb F_q^{d*}}|E\cap l_x|$ was  one  of the most important ingredients in proving Propositions \ref{ProAlex}, \ref{Prosphere}, and \ref{SphereDistance}. One of the purposes of this paper is to announce that  such an idea enables us to extend results of aforementioned propositions to the general dot product set $\Pi(E,F).$ In particular, we prove that  the conclusion of Proposition \ref{Prosphere} still holds in the case when the unit sphere $S_1$ is replaced by the paraboloid $P:=\{x\in \mathbb F_q^d: x_1^2+ \cdots+ x_{d-1}^2 = x_d\}.$ Furthermore, we observe that  if one of the sets $E,F$ is a Salem set,  then  we are able to obtain extremely good results on the generalized dot product set problem.\\

The other purpose of this paper is to introduce a new point of view in deriving the results on generalized dot product sets . Roughly speaking,  we relate the dot product problem to  estimation of the number of lines containing both the origin and an element in a set $E\setminus \{(0,\ldots,0)\} \subset \mathbb F_q^d.$ As a result, we improve statements of  aforementioned propositions for  general two sets $E, F\subset  \mathbb F_q^d.$ In addition,  we classify certain class of  the sets $E, F \subset \mathbb F_q^d$ which yield much better result than that of Proposition \ref{ProAlex}. For example, assuming that  the number of lines both passing through the origin and intersecting with $E\setminus \{(0,\ldots,0)\}$ (or $F\setminus\{(0,\ldots,0)\}$) is much greater than $|E|/q$ (or $|F|/q$),  we shall see that  the result of Proposition \ref{ProAlex} can be improved.

\section{ Preliminaries} 
Discrete Fourier analysis is considered as one of the most useful tools in studying problems in the finite field setting.
In this section, we briefly review it and derive lemmas which are essential in proving our results.

\subsection{Discrete Fourier analysis}
We shall denote by $\psi$ a  nontrivial additive character of $\mathbb F_q.$  All results in this paper are independent of the choice of the character $\psi.$
Recall that $\psi: \mathbb F_q \to \{u\in \mathbb C: |u|=1\}$ is a group homomorphism. 
The orthogonality relation of $\psi$ yields that
$$ \sum_{x\in \mathbb F_q^d} \psi(m\cdot x)=\left\{ \begin{array}{ll} 0  \quad&\mbox{if}~~ m\neq (0,\dots,0)\\
q^d  \quad &\mbox{if} ~~m= (0,\dots,0), \end{array}\right.$$
where  $m\cdot x$ denotes the usual dot-product notation.
Given a function $f: \mathbb F_q^d \rightarrow \mathbb C,$  the Fourier transform of the function $f$ is defined by
$$ \widehat{f}(m)= \frac{1}{q^d} \sum_{x\in \mathbb F_q^d} f(x) \psi(-x\cdot m)\quad \mbox{for}~~m\in \mathbb F_q^d.$$
Then  the Plancherel theorem in this content says that
$$ \sum_{m\in \mathbb F_q^d} |\widehat{f}(m)|^2 = \frac{1}{q^d} \sum_{x\in \mathbb F_q^d} |f(x)|^2.$$
Thus, it is clear that if $E \subset \mathbb F_q^d,$ then
$$ \sum_{m\in \mathbb F_q^d} |\widehat{E}(m)|^2 = \frac{|E|}{q^d}.$$
Here, throughout this paper, we identify the set $E \subset \mathbb F_q^d$ with the characteristic function on the set $E.$

\subsection{Key lemmas related to a general dot product set $\Pi(E,F)$}
Given $E, F\mathbb \subset F_q^d,$  a counting function $\nu$ on $\mathbb F_q$ is defined by
$$\nu(t)=|\{(x,y)\subset E\times F: x\cdot y=t\}|.$$ 
By the definition of the dot product set $\Pi(E,F),$ it is clear that
 $$ |E||F|=\sum_{t\in \Pi(E,F)} 1\times \nu(t).$$
Applying the Cauchy-Schwarz inequality, we see that
\begin{equation}\label{setformula} |\Pi(E,F)|\ge \frac{|E|^2|F|^2}{\sum_{t\in \mathbb F_q} \nu^2(t)}.\end{equation}
Following the argument in \cite{HIKR10}, we obtain the following formula.
\begin{lemma}\label{keylem1} Let $E, F\subset \mathbb F_q^d$ with $(0,\ldots,0)\notin E.$ Then we have
$$ |\Pi(E,F)|\gg_c  \min\left\{ q, ~ \frac{|E||F|^2}{ q^{2d-1}\sum\limits_{x\in \mathbb F_q^{d*}} \sum\limits_{s\in\mathbb  F_q^*}E(sx) |\widehat{F}(x)|^2}\right\},$$
where $\mathbb F_q^{d*} := \mathbb F_q^d \setminus \{(0,\ldots, 0)\}.$
\end{lemma}
\begin{proof} Since $(0,\ldots,0) \notin E,$    it is enough by  (\ref{setformula}) to show that
$$ \sum_{t\in \mathbb F_q} \nu^2(t) \leq \frac{|E|^2|F|^2}{q} + q^{2d-1}|E| \sum_{x\in \mathbb F_q^d} \sum_{s\in \mathbb F_q^*} E(sx) |\widehat{F}(x)|^2.$$
By the Cauchy-Schwarz inequality, it follows that for each $t\in \mathbb F_q$, 
\begin{align*}\nu^2(t)&=\left(\sum_{x\in E} \sum_{y\in F: x\cdot y=t} 1\right)^2 \leq |E| \sum_{x\in E}\left(\sum_{y\in F: x\cdot y=t} 1\right)^2\\
&=|E| \sum_{x\cdot y=t=x\cdot y^\prime} E(x)F(y)F(y^\prime).\end{align*}
Summing over $t\in \mathbb F_q$ and using the orthogonality relation of $\psi$, it follows
\begin{align*}
\sum_{t\in \mathbb F_q} \nu^2(t)&\leq |E|q^{-1} \sum_{s\in \mathbb F_q} \sum_{x\in E, y, y^\prime\in F} \psi(sx\cdot (y-y^\prime))\\
&=q^{-1}|E|^2|F|^2 +  |E|q^{-1} \sum_{s\in \mathbb F_q^*} \sum_{x\in E, y, y^\prime\in F} \psi(sx\cdot (y-y^\prime)).
\end{align*}
By the definition of the Fourier transform and a change of variables, 
\begin{align*}
\sum_{t\in \mathbb F_q} \nu^2(t)&\leq q^{-1}|E|^2|F|^2 + q^{2d-1}|E| \sum_{x\in \mathbb F_q^d, s\in \mathbb F_q^*} E(x) |\widehat{F}(sx)|^2\\
&=q^{-1}|E|^2|F|^2 + q^{2d-1}|E| \sum_{x\in \mathbb F_q^d, s\in \mathbb F_q^*} E(sx) |\widehat{F}(x)|^2,\end{align*}
which completes the proof.
\end{proof}  

\begin{definition} For $E, F\subset \mathbb F_q^d,$ we define
\begin{align*}{\mathfrak B}(E,F)&= \sum_{x\in \mathbb F_q^{d*}} \sum_{s\in\mathbb  F_q^*}E(sx) |\widehat{F}(x)|^2\\
&= \sum_{x\in \mathbb F_q^{d*}} |E\cap l_x| |\widehat{F}(x)|^2, \end{align*}
where $l_x$ is defined as in (\ref{linedef}).
\end{definition} 
According to Lemma \ref{keylem1},  a lower bound of $|\Pi(E,F)|$ can be determined by an upper bound of ${\mathfrak B} (E,F).$  
 More precisely we have the following result.

\begin{lemma}\label{mainlem1} Let $E, F\subset \mathbb F_q^d.$ Assume that $\max_{x\in \mathbb F_q^{d*}} |E\cap l_x| \ll_C q^\beta$ for some $0\le \beta \le 1.$ Then if $|E||F|\gg_C q^{d+\beta}$, we have
$$|\Pi(E,F)|\gg_c q.$$
\end{lemma}
\begin{proof} Without a loss of generality, we may assume that $(0,\ldots,0)\notin E.$ Since $\max_{x\in \mathbb F_q^{d*}} |E\cap l_x| \ll_C q^\beta,$ it follows from the Plancherel theorem that
$$ {\mathfrak B}(E,F)\ll_C q^\beta \sum_{x\in \mathbb F_q^{d}} |\widehat{F}(x)|^2 =q^{\beta-d}|F|.$$
Combining this with Lemma \ref{keylem1}, we conclude that
 $$ |\Pi(E,F)|\gg_c \min \left\{ q, ~ \frac{|E||F|}{q^{d+\beta-1}} \right\},$$
which implies the statement of the lemma.
\end{proof}
\section{Results on the generalized dot product sets}

In this section, we first collect results on the generalized dot product set, which can be obtained by a direct application of Lemma \ref{mainlem1} or Lemma \ref{keylem1}.
For example,  we will be able to simply generalize the results of Propositions \ref{ProAlex}, \ref{Prosphere}, and \ref{SphereDistance}.
As a core part of this section, we derive a dot product result on subsets of the paraboloid, which may not be obtained by a direct application of Lemma \ref{mainlem1}.

\subsection{Direct consequences of Lemma \ref{mainlem1}}

The general version of Proposition \ref{ProAlex} is as follows. 

\begin{theorem}\label{ThmAlex}
 If $E, F \subset \mathbb F_q^d$ with $|E||F|\gg_C q^{d+1}$, then $|\Pi(E,F)|\gg_c q.$
\end{theorem}
\begin{proof}
Since every line  contains exactly $q$ points, it is clear that 
$$\max_{x\in \mathbb F_q^{d*}} |E\cap l_x| \leq q.$$
Thus,  the result follows immediately by using Lemma \ref{mainlem1} with $\beta=1.$
\end{proof}

The following theorem is a generalization of Proposition \ref{Prosphere}.
\begin{theorem}\label{Thmsphere}
Let $F\subset \mathbb F_q^d$ and  $E\subset S_j:=\{x\in \mathbb F_q^d: x_1^2+ \cdots+x_d^2=j\in \mathbb F_q^*\}$. Then if 
 $|E||F|\gg_C q^{d},$ we have  $|\Pi(E,F)|\gg_c q.$
\end{theorem}
\begin{proof} Since $j\neq 0,$ it follows that 
$$ \max_{x\in \mathbb F_q^{d*}} |E\cap l_x| \le 2.$$
Therefore, the statement of the theorem follows by applying Lemma \ref{mainlem1} with $\beta=0.$
\end{proof}

We now give the generalization of Proposition \ref{SphereDistance} .
\begin{theorem}\label{SphereDistance1} Let $S_j=\{x\in \mathbb F_q^d: x_1^2+x_2^2+ \cdots + x_d^2=j\}.$  Suppose that $E\subset S_i $ and $F\subset S_j$ for some $i, j \in \mathbb F_q^*.$
Then if $d\geq 3,$ and $|E||F|\gg_C q^{d},$ we have  $|\cD(E,F)|\gg_c q.$
\end{theorem}
\begin{proof}  Notice that if $x\in E \subset S_i$ and $y\in F \subset S_j$, then
$$\|x-y\|=x\cdot x -2x\cdot y + y\cdot y=i+j-2x\cdot y.$$
Thus, $|D(E,F)|=|\Pi(E,F)|.$ It therefore suffices to prove that $|\Pi(E,F)|\gg_c q$ as long as $|E||F|\gg_C q^d.$ However, this follows immediately from Theorem \ref{Thmsphere}.
\end{proof}

Now we address a result on the dot product set $\Pi(E,F)$ in the case when one of sets $E, F\subset \mathbb F_q^d$ is a Salem set.
Recall that a set $F\subset \mathbb F_q^d$ is a Salem set if $ |\widehat{F}(m)|\ll_C q^{-d}\sqrt{|F|}$ for all $m\neq (0,\ldots,0).$
\begin{theorem} Let $E,F\subset \mathbb F_q^d.$ If $F$ is a Salem set and $|F|\gg_C q,$ then 
$$|\Pi(E,F)|\gg_c q.$$ \end{theorem}
\begin{proof} Notice that we may assume that $(0,\ldots, 0)\notin E.$ Since $F$ is a Salem set, we see that 
$$\max_{x\in \mathbb F_q^{d*}}|\widehat{F}(x)|^2 \ll_C q^{-2d} |F|.$$
It therefore follows that
$$ {\mathfrak B}(E,F) \ll_C q^{-2d} |F|  \sum_{x\in \mathbb F_q^{d*}} \sum_{s\in\mathbb  F_q^*}E(sx)<q^{-2d} |F| |E|q =q^{1-2d} |E||F|.$$
By this and Lemma \ref{keylem1}, we obtain that
$$  |\Pi(E,F)|\gg_c  \min\{ q, ~|F|\},$$
which implies the statement of the theorem.
\end{proof}

\subsection{Dot product sets determined by subsets of the  paraboloid}
In the finite field setting, the paraboloid in $\mathbb F_q^d$, denoted by $P$, is defined by
$$P=\{x\in \mathbb F_q^d: x_1^2+\cdots+ x_{d-1}^2=x_d\},$$
which is an analog of the Euclidean paraboloid. \\

Unlike the sphere $S_j$ with nonzero radius, the paraboloid $P\subset \mathbb F_q^d, d\geq 3,$ contains lines through the origin.
For example,  the set $H:=\{x\in P: x_d=0\}$ consists of some of lines through the origin.
Thus, if $E\subset P$ contains some of such lines, then $\max_{x\in \mathbb F_q^{d*}} |E\cap l_x| =q-1.$
In this case, if we simply use Lemma \ref{mainlem1} , then we only get that if $|E| |F|\gg_C q^{d+1},$ then $|\Pi(E,F)|\gg_c q.$
This result is much weaker than the dot product result on spheres with nonzero radius, but there are no known good results for sets in the paraboloid.
In this subsection, we prove that if $E, F \subset P$ and $|E||F|\gg_C q^d,$ then $|\Pi(E,F)|\gg_c q.$
We begin with a definition.
Let $\pi:\mathbb F_q^d \to \mathbb F_q^{d-1}$ be a projection map defined as
$$\pi(x)=(x_1, \ldots, x_{d-1})\quad\mbox{for}~~x=(x_1,\ldots, x_{d-1}, x_d).$$
We have the following result.
\begin{lemma} \label{caselem} Let $E\subset P$ and $F\subset \mathbb F_q^d.$
If $|E||\pi(F)|\gg_C q^d,$ then  $|\Pi(E,F)|\gg_c q.$
\end{lemma}
\begin{proof}Write that
$E=G\cup B$ where
$$G=\{x\in E: x_d\neq 0\}\quad\mbox{and}\quad B=\{x\in E: x_d=0\}.$$
We may assume that either $|G|\geq |E|/2$ or $|B|\geq |E|/2.$\\

{\bf Case 1.} Assume that $|G|\geq |E|/2.$
Since $G\subset P$, it is not hard to see that $|G\cap l_x|\leq 1 $ for all $x\in \mathbb F_q^{d*}.$
By Lemma \ref{mainlem1}, we see that 
if $|G||F|\gg_C q^d,$ then $|\Pi(G, F)|\gg_c q.$
Since $2|G||F|\geq |E| |\pi(F)|\gg_C q^d,$ and $|\Pi(E,F)| \geq |\Pi(G,F)|,$  the statement of the lemma follows.\\

{\bf Case 2.} Assume that $|B|\geq |E|/2.$ By the definitions of $B$ and the dot product, notice that
\begin{align*}\Pi(B, F)&=\Pi(B, \pi(F)\times \{0\})\\
                                &=\Pi(\pi(B), \pi(F)):=\{\alpha\cdot \beta\in \mathbb F_q: \alpha\in \pi(B)\subset \mathbb F_q^{d-1},  ~\beta\in \pi(F) \subset \mathbb F_q^{d-1} \}\end{align*}
Since $\pi(B), \pi(F) \subset \mathbb F_q^{d-1},$  we can use Theorem \ref{ThmAlex} for dimension $d-1$ to deduce that
$$ |\Pi(B,F)|=|\Pi( \pi(B), \pi(F))| \gg_c q \quad\mbox{if}~~ |\pi(B)||\pi(F)|\gg_C q^d.$$
Since $B$ is a subset of the paraboloid $P,$  it is clear that $|\pi(B)|=|B|\geq |E|/2,$ 
where the inequality follows by our case assumption. Since $|\Pi(E,F)|\geq |\Pi(B,F)|,$ we complete the proof.
\end{proof}

Since $|\Pi(F)|=|F| $ for $F\subset P,$ the following result follows immediately from Lemma \ref{caselem}.
\begin{theorem} Let $E, F \subset P\subset \mathbb F_q^d.$
If $|E||F|\gg_C q^d,$ then $|\Pi(E,F)|\gg_c q.$
\end{theorem}

\begin{remark}\label{Rem1} Let $E\subset P\subset \mathbb F_q^d,$ and $F\subset \mathbb F_q^d$ with $|\pi(F)|\gg_c |F|/q^\gamma$ for some $0\leq \gamma \leq 1.$
In this case, Lemma \ref{caselem} implies that if $|E||F|\gg_C q^{d+\gamma},$ then $|\Pi(E,F)|\gg_c q.$ \end{remark}
Here we may have a natural  question. 
\begin{question}\label{question1}
 Let $E\subset P$ and $F\subset \mathbb F_q^d.$ Then is it true that $|\Pi(E,F)|\gg_c q $ as long as $|E||F|\gg_c q^d?$\end{question}

Considering Remark \ref{Rem1}, it seems that the answer is negative.
However, Theorem \ref{Thmsphere} says that if we replace the paraboloid $P$ by the sphere $S_j$ with nonzero radius, then the answer is positive.
Now, we show that if the paraboloid $P$ is appropriately translated, then the answer to Question \ref{question1} is also positive.
\begin{theorem} \label{Pt}Let $a \in\mathbb F_q^d \setminus \overline{P}:=\{x\in \mathbb F_q^d: x_1^2+\cdots+x_{d-1}^2=-x_d\}.$  Suppose that $E\subset P+a:=\{x+a: x\in P\}$ and $F\subset \mathbb F_q^d.$
Then if $|E||F|\gg_C q^d,$ we have $|\Pi(E,F)|\gg_c q.$\end{theorem}

\begin{proof} By Lemma \ref{mainlem1}, it suffices to prove that for every $a\in \mathbb F_q^d\setminus \overline{P},$ and $x\neq (0,\ldots,0),$
\begin{equation}\label{equ2} | (P+a)\cap l_x|\ll_C 1,\end{equation}
where we recall that $\l_x=\{s x\in \mathbb F_q^{d*}: s\in \mathbb F_q^*\}.$
Fix $x\in P+a.$ 
Then it follows that
$$ (x_1-a_1)^2 + \cdots+ (x_{d-1}-a_{d-1})^2=x_d -a_d.$$
With this assumption, it is enough to prove that
$$|\{s\in \mathbb F_q^*: sx\in P+a\}|\leq 2.$$
 It follows from a routine algebra  that if $a\notin \overline{P}$( namely, $a_1^2+\cdots+a_{d-1}^2\ne -a_d$), then
$$ |\{s\in \mathbb F_q^*: (sx_1-a_1)^2+\cdots+ (sx_{d-1}-a_{d-1})^2= sx_d -a_d\}| \leq 2.$$
Thus, the proof is complete.
\end{proof}

Observe that $(0, \ldots, 0) \in P,$ but the sphere $S_j$ for $j\ne 0$ or $P+a$ for $a\notin P$ does not contain $(0,\ldots,0).$
From Theorem \ref{Thmsphere} and Theorem \ref{Pt}, this observation may lead us to the following conjecture.
\begin{conjecture} 
Let  $ V=\{x\in \mathbb F_q^d: Q(x)=0\}$ be a variety where  $Q(x) \in \mathbb F_q[x_1,\ldots,x_d]$ is a polynomial.
In addition, assume that $(0,\ldots, 0) \notin V.$
If $E\subset V$ and $F\subset \mathbb F_q^d $ with $|E||F|\gg_C q^d,$ then we have
$$|\Pi(E,F)|\gg_c q.$$
\end{conjecture}

\section{Sharpened results on the generalized dot product set}
In the previous section, we deduced the results on the dot product set by considering $\max_{x\in \mathbb F_q^{d*}} |E\cap l_x|.$
This method may give us a sharp result for the set $E\subset \mathbb F_q^d$ in the case when $|E\cap l_x| \sim |E\cap l_y|$ for almost  elements $x,y \in \mathbb F_q^{d*}.$
However, it may not be efficient in the case when the variation of $|E\cap \l_x|$ for $x\in \mathbb F_q^{d*}$ is relatively large.
In this section, we introduce a new approach to compensate the defect of the previous method
and provide improved statements of the results in the previous section.\\

Now, we derive a new formula to determine $|\Pi(E,F)|,$ which is much stronger than Lemma \ref{mainlem1}.
\begin{lemma}\label{mainlem2} Let $E,F\subset \mathbb F_q^d.$
Assume that the number of lines through the origin and a point in $E\setminus \{(0,\ldots,0)\}$ is at least $\sim q^{-\alpha}|E|$ for some $0\leq \alpha\leq 1.$
If $|E||F|\gg_C q^{d+\alpha}$ then there exists a set $E_0\subset E$ with $|E_0|\sim q^{-\alpha}|E|$ such that
$$ |\Pi(E_0, F)|\gg_c q.$$
\end{lemma}
\begin{proof} Note that we may assume that $(0,\ldots,0)\notin E.$
Let $n$ be an integer with $n\sim  q^{-\alpha}|E|.$ By assumption, we may choose $n$ lines, say that $l_j, j=1,2,\ldots, n,$ such that each of them contains at least one point in $E$, and is also passing through the origin. 
For each $j=1,2,\ldots, n$, choose exactly an element $x^j \in l_j \cap E$ and define 
$$ E_0=\{x^j: j=1,2,\ldots, n\}.$$
Since $|E_0|=n\sim q^{-\alpha}|E|$ for some $0\leq \alpha\leq 1,$
it suffices to prove that $ |\Pi(E_0, F)|\gg_c q$ as long as $|E||F|\gg_C q^{d+\alpha}.$
By the definition of $E_0,$ it is clear that $\sum_{s\in\mathbb F_q^*} E_0(sx) =1$ for each $x\ne (0,\ldots,0).$
This implies that ${\mathfrak B} (E_0,F) \leq \sum_{x\in \mathbb F_q^d} |\widehat{F}(x)|^2=q^{-d}|F|.$
Now applying Lemma \ref{keylem1} with $E_0, F$ yields that
$$ |\Pi(E_0, F)|\gg_c \min\left\{q,~~ \frac{|E_0||F|}{q^{d-1}} \right\}.$$
Since $|E_0|\sim q^{-\alpha}|E|$ , the statement of the theorem follows immediately from the assumption that $|E||F|\gg_C q^{d+\alpha}.$
\end{proof}

The value $\alpha$ given in Lemma \ref{mainlem2} must be contained in $[0,1].$ For example, if $E$ lies on a unit sphere $S_1:=\{x\in \mathbb F_q^d: \|x\|=1\},$ then $\alpha$ can be taken as zero.
In addition, observe that for each $E \setminus\{(0,\ldots,0)\}$, there are at least $\sim q^{-1}|E|$  such lines, because a line contains exactly $q$ points. Namely, $\alpha$ must be less than or equal to one. 
Also notice from Lemma \ref{mainlem2} that we can expect better dot product results whenever the set $E$ intersects with lots of such lines. In order words, the smaller $\alpha$ is, the better the result is.
As mentioned before, the  $(d+1)/2$ is the optimal exponent to obtain the conclusion of   Proposition \ref{ProAlex} for arbitrary set $E.$ 
Thus, the exponent $d+1$ in the assumption of  Theorem \ref{ThmAlex} is also optimal in general.
However, Lemma \ref{mainlem2} illustrates that  the exponent $d+1$ can be improved in the case when the set $E\setminus \{(0,\ldots,0)$ intersects with at least $|E|/q^{1-\varepsilon}$ lines through the origin for some $0<\varepsilon\leq 1.$\\

Now, we claim that Lemma \ref{mainlem2} is much superior to Lemma \ref{mainlem1}.
Indeed,  an upgraded version of Lemma \ref{mainlem1} can be given by a corollary of Lemma \ref{mainlem2}. More precisely, we can derive the following fact.
\begin{lemma}\label{mainlem3} Let $E, F\subset \mathbb F_q^d.$ Assume that $\max_{x\in \mathbb F_q^{d*}} |E\cap l_x| \ll_C q^\beta$ for some $0\le \beta \le 1.$ Then  if $|E||F|\gg_C q^{d+\beta},$
 there exists a set $E_0\subset E$ with $|E_0|\sim q^{-\beta}|E|$ such that
$$ |\Pi(E_0, F)|\gg_c q.$$
\end{lemma}
\begin{proof} Since  $\max_{x\in \mathbb F_q^{d*}} |E\cap l_x| \ll_C q^\beta,$ it is clear that  the number of lines through the origin and a point in $E\setminus \{(0,\ldots,0)\}$ is at least $\sim q^{-\beta}|E|.$
Hence, the statement of the lemma follows immediately by Lemma \ref{mainlem2}.
\end{proof}
  Lemma \ref{mainlem3} enables us to deduce stronger conclusion than  Lemma \ref{keylem1}, because $  |\Pi(E_0, F)| \leq |\Pi(E,F)|$ for $E_0\subset E.$ 
For example,  Theorem \ref{ThmAlex}  can be improved by the following statement.
\begin{theorem}\label{sharpen}
 Let $E, F\subset \mathbb F_q^d.$
It $|E||F|\gg_C q^{d+1},$ then there exists a set $E_0 \subset E$ with $|E_0|\sim q^{-1}|E| $ such that
$$ |\Pi(E_0, F)|\gg_c q.$$
\end{theorem} 
\begin{proof} Since $ |E\cap l_x| \leq q,$  this theorem is an immediate consequence of Lemma \ref{mainlem3}.
\end{proof}
Notice that  Lemma \ref{mainlem3} can be also used to deduce the improved conclusions of Theorem \ref{Thmsphere} and Theorem \ref{SphereDistance}.
We close this paper with an important remark on Theorem \ref{sharpen}.
\begin{remark} The authors in \cite{CEHIK09} studied the pinned distance sets and proved the following strong result (Theorem 2.2 in \cite{CEHIK09}).
\begin{proposition}\label{goodbut}  Let $E\subset \mathbb F_q^d, d\geq 2.$ If $|E|\ge q^{(d+1)/2},$ then there exists a set $E^\prime \subset E$ with $|E^\prime|\gg_c |E|$ such that
if $x\in E^\prime,$ then $|\Pi(x,E)|> q/2,$ where $\Pi(x,E):=\{x\cdot y: y\in E\}.$ \end{proposition}
This proposition is much superior to our Theorem \ref{sharpen} in the case when $E=F.$
The existence of such set $E^\prime$ in Proposition \ref{goodbut} was proved by using an averaging argument. Therefore, there is no information about how to choose an exact element $x$ of $E^\prime$ so that
$ |\Pi(x, E)|\gg_c q.$ 
On the other hand,  the proof of our Theorem \ref{sharpen} clearly indicates how to choose the set $E_0$. In practice, our Theorem \ref{sharpen} can be very useful.
\end{remark}

\end{document}